\documentclass[11pt]{amsproc}
\usepackage{amssymb}
\usepackage{amsfonts}
\usepackage{xcolor}

\theoremstyle{plain}
\newtheorem{theorem}{Theorem}

\newtheorem{claim}[theorem]{Claim}

\newtheorem{conjecture}[theorem]{Conjecture}

\newtheorem{lemma}[theorem]{Lemma}

\numberwithin{equation}{section}

\begin{document}
\title[Locally irregular decompositions of regular graphs]
{Decomposability of regular graphs to $4$ locally irregular subgraphs}
\author{Jakub Przyby{\l}o}
\address{AGH University of Krakow, al. A. Mickiewicza
30, 30-059 Krakow, Poland}
\email{jakubprz@agh.edu.pl}

\begin{abstract}
A locally irregular graph is a graph whose adjacent vertices have distinct degrees. It was conjectured that every connected graph is edge decomposable to $3$ locally irregular subgraphs, unless it belongs to a certain family of exceptions, including graphs of small maximum degrees, which are not decomposable to any number of such subgraphs. Recently Sedlar and \v{S}krekovski exhibited a counterexample to the conjecture, which necessitates a decomposition to (at least) $4$ locally irregular subgraphs. We prove that every $d$-regular graph with $d$ large enough, i.e. $d\geq 54000$, is decomposable to $4$ locally irregular subgraphs. Our proof relies on a mixture of a numerically optimized application of the probabilistic method and certain deterministic results on degree constrained subgraphs due to Addario-Berry, Dalal, McDiarmid, Reed, and Thomason, and to Alon and Wei, introduced in the context of related problems concerning irregular subgraphs.
\end{abstract}

\maketitle

\section{Introduction}

It is one of the most basic observations in graph theory that no simple, finite graph $G$ with at least $2$ vertices can be completely irregular, i.e. its vertices cannot have all pairwise distinct degrees. Some alternative concepts of irregular graphs were thus investigated already by Chartrand, Erd\H{o}s and Oellermann~\cite{ChartrandErdosOellermann}, and others, see e.g.~\cite{k-PathIrregular,HighlyIrregular}. 
More recently, a notion of a \emph{locally irregular graph} was introduced in~\cite{LocalIrreg_1}.
It is a graph $G=(V,E)$ whose adjacent vertices have distinct degrees, i.e. $d_G(u)\neq d_G(v)$ for every edge $uv\in E$. In~\cite{LocalIrreg_1} also a problem of edge decomposability of graphs to locally irregular subgraphs was raised. We say that a graph $G=(V,E)$ is \emph{decomposable} to $k$ locally irregular subgraphs if its edge set can be partitioned to $k$ subsets: $E=E_1\cup E_2\cup\ldots\cup E_k$ such that the subgraph $G_i=(V,E_i)$ is locally irregular for $i=1,2,\ldots,k$.

This concept is related with a recent direction of research of Alon and Wei~\cite{Alon-Wei}, also developed by Fox,  Luo and Pham~\cite{FoxIrregular}, who investigated so-called \emph{irregular subgraphs}. 
See~\cite{PrzybyloAlmostIrregularAMC} for discussion and results linking the two concepts.
The notion of locally irregular graphs is also  strongly related with yet another well-known descendant of research on irregular graphs~\cite{ChartrandErdosOellermann} and so-called irregularity strength of graphs, see e.g.~\cite{Chartrand},  which is the \emph{1--2--3 Conjecture}. This was posed in 2004 by
Karo\'nski,  {\L}uczak and Thomason~\cite{123KLT} and recently confirmed by Keusch~\cite{Keusch2}. 
It may be phrased as follows. Given any connected graph with at least $3$ vertices, one may multiply every its edge $e$ to at most $3$ copies of $e$ so that the resulting multigraph is locally irregular, i.e. its adjacent vertices have pairwise distinct degrees. 
See e.g.~\cite{LocalIrreg_1} for a more comprehensive discussion on relations between these concepts. 

Also in~\cite{LocalIrreg_1}  a family $\mathfrak{T}'$ of all connected graphs which are not decomposable to any number of locally irregular subgraphs was identified. The central initial problem in the study devoted to locally irregular graphs was the following conjecture.
\begin{conjecture}[\cite{LocalIrreg_1}]\label{LIDConjecture1}
Each connected graph $G$ which does not belong to $\mathfrak{T}'$ is decomposable to $3$ locally irregular subgraphs.
\end{conjecture}
It is worth mentioning that the family $\mathfrak{T}'$ includes only graphs with maximum degree upper bounded by $3$. It thus was natural to investigate wether graphs with sufficiently large degrees are decomposable to $3$ locally irregular subgraphs. 
In~\cite{LocalIrreg_1} this was in particular proven to hold for regular graphs, which are in some sense as far from being locally irregular as possible.
\begin{theorem}[\cite{LocalIrreg_1}]\label{Theorem107}
Each $d$-regular graph with $d\geq 10^7$ is decomposable to $3$ locally irregular subgraphs.
\end{theorem}
This result was further extended towards all graphs $G$ with minimum degree $\delta(G)\geq 10^{10}$ in~\cite{LocalIrreg_2}. 
Exploiting in particular this theorem, Bensmail, Merker and Thomassen~\cite{BensmailMerkerThomassen}
proved next that each connected graph outside  $\mathfrak{T}'$ is decomposable to at most $328$ locally irregular subgraphs. The bound of $328$ was also latter refined to $220$ in~\cite{LocalIrreg_Cubic}, see e.g.~\cite{LocalIrreg_1,BaudonBensmailSopena-loc_irr_compl,BensmailDrossNisse-loc_irreg,LeiLianShiZhao-loc_irreg,LintzmayerEtAl-loc_irreg,LocalIrreg_Cubic,SedlarSkrekovski-LocIrreg2}
for other related results concerning locally irregular graphs.

Surprisingly, quite recently a single connected exception to Conjecture~\ref{LIDConjecture1} was found 
by Sedlar and \v{S}krekovski~\cite{SedlarSkrekovski-LocIrreg}. This is decomposable to $4$ locally irregular subgraphs. Consequently, they posed the following less restrictive variant of the conjecture.
\begin{conjecture}[\cite{SedlarSkrekovski-LocIrreg}]\label{LIDConjecture2}
Each connected graph $G$ which does not belong to $\mathfrak{T}'$ is decomposable to $4$ locally irregular subgraphs.
\end{conjecture}

As the main contribution of this paper we provide the following significant reduction of the upper bound $10^7$ from Theorem~\ref{Theorem107} in context of the new Conjecture~\ref{LIDConjecture2}. In a way this concerns the most challenging in face of our goals family of graphs, which are extremely far from being (locally) irregular.
\begin{theorem}\label{MainResult-4-decomposability_of_regular}
Each $d$-regular graph with $d\geq 54000$ is decomposable to $4$ locally irregular subgraphs.
\end{theorem}
Let us remark that our approach is  significantly different from the one used in~\cite{LocalIrreg_1} to prove Theorem~\ref{Theorem107}. It benefits in particular, but not exclusively, from the  allowed one more locally irregular subgraph, which admitted an entirely new design of a desired decomposition.
Its first phase exploits the probabilistic method in order to initially decompose a given regular graph to $2$ subgraphs with relatively small chromatic numbers, where we put some effort to optimize our approach with regard to minimizing the lower bound for $d$ in Theorem~\ref{MainResult-4-decomposability_of_regular}. In the second, deterministic part of our argument we also make use of results stemming from the mentioned related problems, i.e. certain theorems on degree constrained subgraphs of Addario-Berry, Dalal, McDiarmid, Reed, and Thomason~\cite{Louigi30}, and of Alon and Wei~\cite{Alon-Wei}.
We believe the applied strategy is quite interesting by itself. 

In the second section we shall outline the main ideas of the proof of Theorem~\ref{MainResult-4-decomposability_of_regular}. Next we shall present several useful tools. Section~\ref{Proof_Main_Section} contains the complete proof of our main result. The last section includes concluding remarks.

\section{Outline of the main ideas} 

The main vague idea of the proof of Theorem~\ref{MainResult-4-decomposability_of_regular} is the following.
In the beginning we shall partition the edges of a given $d$-regular graph $G=(V,E)$ to two subsets. Each of these shall induce a graph $H_i$, $i=0,1$, with several identical features. 
First of all each $H_i$ shall be close to a $\frac{d}{2}$-regular graph, but shall have relatively small chromatic number. For this purpose we shall randomly associate to every vertex $v\in V$ a pair of integers $(O_v,I_v)\in [K]^2$, called its colour, where $K$ is small compared to $d$ and $[K]=\{1,2,\ldots,K\}$. Than we shall colour the edges accordingly with $0$ and $1$, following several rules -- the edges coloured $0$ shall induce $H_0$ and the ones coloured $1$ shall form $H_1$. 
The values $O_v$ shall in fact associate the vertices with independent sets in $H_0$ (witnessing its relatively small chromatic number) and the values $I_v$ shall play the same role for $H_1$. In order to maintain a better error control, the aforementioned $0$-$1$-edge partition shall be partly performed deterministically, using Lemma~\ref{Alon-Wei_lemma} of Alon and Wei. 
More specifically, the random vertex colours assignment shall be mainly used to provide large independent sets in $H_0$ and $H_1$, while the rest of the required features of $H_0$, $H_1$ shall be assured due to Lemma $9$, providing almost perfectly equal distribution of colours $0$ and $1$ in certain distinguished  subsets of edges.

Then each of the graphs $H_i$, $i=0,1$, shall be further decomposed into two subgraphs, say $F_1$ and $F_2$. These, unlike $H_0$ and $H_1$, shall have different sets of properties and features, 
which in varied ways shall imply that $F_1$, $F_2$ are locally irregular.
 The degrees in $F_1$ shall be determined by $O_v$ or $I_v$. More precisely, these shall be (almost) equal to twice one of these values modulo $2K$, and thus shall be different for adjacent vertices from different independent sets mentioned above. In order to construct such $F_1$ we shall use Lemma~\ref{1_6Lemma}, concerning existence in every graph very restrictively degree constrained  subgraphs, which essentially was 
developed by Addario-Berry, Dalal, McDiarmid, Reed and Thomason while conducting research over the 1--2--3 Conjecture. In order to assure that the complement of $F_1$ in $H_i$, that is $F_2$, is locally irregular as well, we shall on the other hand first identify potentially risky edges from the point of view of $F_2$. 
  All of these shall be included in $F_1$ prior to using Lemma~\ref{1_6Lemma}, which shall allow us to avoid degree conflicts in $F_2$. 

As the values $(O_v,I_v)$ shall be chosen randomly and independently for every $v\in V$, some pairs of adjacent vertices shall receive identical colours. If this occurs we shall uncolour both of them. We shall however show that the set $U$ of such uncolored vertices is relatively sparse, and thus we shall be able to handle them separately, using a natural greedy approach. In fact these vertices shall be dealt with before the remaining ones.

\section{Tools}

We shall use several tools useful while applying the probabilistic method.
In particular some concentration inequalities: 
the well-known Chernoff Bound and the McDiarmid's Inequality, cf. e.g.~\cite{FriezeKaronski}.
\begin{theorem}[\textbf{Chernoff Bound}]
Let $X_1,X_2,\ldots,X_n$ be independent random variables taking values in $\{0,1\}$. Let $X=\sum_{i=1}^nX_i$ and $\mu=\mathbb{E}(X)$. Then for all $\delta\geq 0$ we have
$$\mathbb{P}\left(X\geq (1+\delta)\mu\right)\leq \left(\frac{e^\delta}{(1+\delta)^{1+\delta}}\right)^\mu.$$
\end{theorem}
Note the same upper bound holds if we only know that  $\mu\geq \mathbb{E}(X)$.
\begin{theorem}[\textbf{McDiarmid's Inequality}]\label{McDiarmidTheorem}
Let $Z = Z(W_1,W_2,\ldots,W_n)$ be a ran\-dom variable that depends on $n$ independent random variables $W_1, W_2,$ $\ldots,$ $W_n$. Suppose
that
$$|Z(W_1,\ldots,W_i,\ldots,W_n)-Z(W_1,\ldots,W'_i,\ldots,W_n)| \leq c_i$$
for all $i = 1,2,\ldots,n$ and $W_1, W_2, \ldots, W_n, W'_i$. Then for all $t > 0$ we have
$$\mathbb{P}\left(Z \geq \mathbb{E}(Z)+t\right) \leq e^{-\frac{t^2}{2\sum_{i=1}^n c_i^2}}.$$
\end{theorem}

We shall also need a variant of the symmetric Lov\'asz Local Lemma, involving conditional probabilities of bad events, cf. e.g.~\cite{AlonSpencer} (Corollary 5.1.2 and comments beneath its proof).
\begin{theorem}[\textbf{Lov\'asz Local Lemma}]
\label{LLL-symmetric-lopsided}
Let $\mathcal{A}=\{A_1,A_2,\ldots,A_n\}$ be a family of 
events in an arbitrary pro\-ba\-bi\-li\-ty space.
Suppose that for every $i\in [n]$ there is a set $\Gamma(A_i)\subset \mathcal{A}$ of size at most $D$ such that for each $\mathcal{B}\subseteq \mathcal{A}\smallsetminus (\Gamma(A_i)\cup \{A_i\})$: 
$$ \mathbb{P}\left(A_i|\bigcap_{B\in\mathcal{B}}\overline{B}\right)  \leq \frac{1}{e(D+1)}.$$
Then $ \mathbb{P}\left(\bigcap_{i=1}^n\overline{A_i}\right)>0$.
\end{theorem}

For the deterministic part of our proof we shall also need the following lemma from~\cite{LocalIrreg_2} (Corollary 9), concerning degree constrained subgraph of a given graph, which is a direct consequence of Theorem 2.2 due to Addario-Berry, Dalal, McDiarmid, Reed and Thomason from~\cite{Louigi30}
(see also \cite{Louigi2,Louigi} for similar degree theorems and their applications).
\begin{lemma} \label{1_6Lemma}
Suppose that for some graph $G=(V,E)$ with minimum degree at least $12$ we have chosen, for every
vertex $v$, an integer $\lambda_v\geq 2$ with $6\lambda_v \leq d(v)$. Then for every assignment
$$t:V\to \mathbb{Z},$$
there exists a spanning subgraph $H$ of $G$ such that $d_H(v)\in[\frac{d(v)}{3},\frac{2d(v)}{3}]$ and
$d_H(v)\equiv t(v) \pmod {\lambda_v}$ or $d_H(v)\equiv t(v)+1 \pmod {\lambda_v}$ for each $v\in V$.
\end{lemma}

We shall also use the following lemma due to Alon and Wei~\cite{Alon-Wei} (see Lemma 4.1).
\begin{lemma}\label{Alon-Wei_lemma} Let $G = (V, E)$ be a graph, and let $z : E \to [0, 1]$ be a weight function assigning to each edge $e \in E$ a real weight $z(e)$ in $[0, 1]$. Then there is a function $x : E \to
\{0,1\}$ assigning to each edge an integer value in $\{0,1\}$ so that for every $v \in V$
$$\sum_{e\ni v}z(e) - 1 <\sum_{e\ni v}x(e)\leq \sum_{e\ni v}z(e)+1.$$  
\end{lemma}

Let $G=(V,E)$ be a graph and let $v,w\in V$, $V'\subseteq V$, $E'\subseteq E$. 
By $N_G(v)$ we understand the set of neighbours of $v$ in $G$.
We denote by $E'(v)$ the set of edges in $E'$ incident with $v$. 
We also define $d_{E'}(v) = |E'(v)|$, $d_{V'}(v)=|N_G(v)\cap V'|$. 
Finally, by ${\rm dist}_G(v,w)$ we denote the distance between $v$ and $w$ in $G$.

\section{Proof of Theorem~\ref{MainResult-4-decomposability_of_regular}}\label{Proof_Main_Section}

\subsection{Random colour assignment and distinguished sets}

Suppose $G$ is a $d$-regular graph with $d\geq 54000$. 
 Let us fix a few optimized constants:
\begin{equation}\label{OptimizedMainConstants}
k=0.025, \hspace{0.5cm}s=0.0031, \hspace{0.5cm}r=0.26, \hspace{0.5cm}u=0.131.
\end{equation}
Set
$$K=\left\lceil kd\right\rceil.$$
Let us randomly assign to every vertex $v\in V$ a colour  $C_v=(O_v,I_v)\in [K]^2$, i.e. $O_ v,I_v$, $v\in V$, are independent random variables uniformly distributed over $[K]$. 
Denote the set of vertices to be uncoloured by
$$U=\{v\in V~|~ \exists w\in N_G(v): C_w=C_v\}.$$
and let $U^e$ be the set of edges induced by $U$,
$$U^e=\{vw\in E:v,w\in U \}.$$
We further distinguish the set of edges ,,\emph{touching}'' $U$: 
$$T=\{vw\in E: v\in U \vee w\in U\}$$
and the set of \emph{special} edges:
$$S=\{vw\in E\smallsetminus T: O_v=O_w \vee I_v=I_w\}.$$
We also define a certain set of \emph{risky edges}, where by $|m-n|_K$ we understand the \emph{distance} between integers $m$ and $n$ \emph{modulo} $K$, i.e. $|m-n|_K=\min\{m-n~({\rm mod}~K),~ n-m~({\rm mod}~K)\}$:
$$R=\left\{vw\in E\smallsetminus T: 1\leq |O_v-O_w|_K\leq \frac{sd+7}{2} \vee 1\leq |I_v-I_w|_K\leq \frac{sd+7}{2}\right\}$$
and the set of edges which are risky but not special: 
$$R'=R\smallsetminus S.$$
Finally, we define two more specific subsets of edges: 
$$E'=E\smallsetminus (T\cup R),\hspace{1.0cm}E''=E\smallsetminus (T\cup R\cup S).$$

We shall first prove that we may choose such a vertex colours assignment so that each vertex $v$ of $G$ is adjacent with bounded numbers of vertices of certain types. For clarity it shall usually be  more convenient to refer further on to the corresponding edges, joining a given vertex $v$ with vertices of some types. The only exception adheres to $d_U(v)$, which shall always mean the number of neighbours $w\in U$ of $v$  such that $vw\in E$ (i.e. shall always refer to the neighbourhood of $v$ in the entire $G$).

\subsection{Assuring desired features of the distinguished sets}

\begin{claim}\label{MainProbClaim}
With positive probability, for every $v\in V$:
\begin{itemize}
\item[(i)]  $d_S(v) <  sd$, 
\item[(ii)]  $d_R(v) < rd$, 
\item[(iii)]  $d_U(v) < ud$. 
\end{itemize}
\end{claim}

\begin{proof}
For every $v\in V$, let us denote the events:
\begin{itemize}
\item $S_v: ~d_S(v)\geq sd$;
\item $R_v: ~d_R(v)\geq rd$;
\item $U_v: ~d_U(v)\geq ud$.
\end{itemize}
Let $\mathcal{A}$ be the set of all such events.
For every $v$ and $A_v\in\{S_v,R_v,U_v\}$ we define: 
\begin{equation}\label{GammaDefinition}
\Gamma(A_v)=\bigcup_{w:~{\rm dist}_G(v,w)\leq 3}\{S_w,R_w,U_w\}\smallsetminus \{A_v\}.
\end{equation}
Note that for every $v$ and $A_v\in\{S_v,R_v,U_v\}$,
\begin{equation}\label{GammaEstimation}
|\Gamma(A_v)|\leq 3(d^3-d^2+d) +2 < 3d^3-1
\end{equation}
and 
\begin{equation}\label{DeterminationOfAv}
A_v  ~{\rm is~ determined~ by~ the~colours}~ C_w~  {\rm of} ~w\in V~{\rm such~ that~ dist}_G(v,w)\leq 2.
\end{equation}

Instead of considering $S$ itself, we shall focus on its superset, including also edges touching $U$:
$$S^*=\{vw\in E: O_v=O_w \vee I_v=I_w\}.$$
We obviously have that $d_S(v)\leq d_{S^*}(v)$ for each vertex $v$. 

Consider any $v\in V$ and any subset $\mathcal{B}\subseteq \mathcal{A}\smallsetminus \left(\Gamma(A_v)\cup\{A_v\}\right)$ where $A_v\in\{S_v,R_v,U_v\}$. As the event: $d_{S^*}(v)\geq sd$ is determined by the choices of $C_w$ for vertices $w$ at distance at most $1$ from $v$, by~\eqref{GammaDefinition} and~\eqref{DeterminationOfAv}, this event is mutually independent of $\mathcal{B}$, i.e. $\mathbb{P}(d_{S^*}(v)\geq sd~|~\bigcap_{B\in\mathcal{B}}\overline{B}) = \mathbb{P}(d_{S^*}(v)\geq sd)$. 
Below we thus bound the latter one. 

Note that if we fix a colour $C_v$ of $v$ then each edge $vw\in E(v)$ is independently included in $S^*(v)$ 
with probability 
$$2\cdot\frac{1}{K}-\frac{1}{K^2}< \frac{2}{kd}.$$
Thus, for any fixed $C\in[K]^2$, 
$$\mathbb{E}\left(d_{S^*}(v)~|~C_v=C\right) < \frac{2}{k}=80< 0.0015d = \mu_s$$
(for $d\geq 54000$).
Set $s_1=0.0015$, $s_2=s-s_1=0.0016$ and $\delta_s=\frac{s_2}{s_1}$.
Thus, $sd=(1+\delta_s)\mu_s$ and by the Chernoff Bound, for any $C\in[K]^2$,  
\begin{eqnarray}
\mathbb{P}\left(d_{S^*}(v) \geq sd~|~C_v=C\right) 
&=& \mathbb{P}\left(d_{S^*}(v)\geq (1+\delta_s)\mu_s~|~C_v=C\right) \nonumber\\
&\leq&  \left(\frac{e^{\delta_s}}{(1+\delta_s)^{1+\delta_s}}\right)^{\mu_s}
= \frac{e^{s_2d}}{\left(\frac{s}{s_1}\right)^{sd}}  \nonumber\\
&=& \left(\frac{e^{s_2}}{\left(\frac{s}{s_1}\right)^{s}}\right)^d 
< \frac{1}{e3d^3}. \label{dS*Conditional}
\end{eqnarray}
In order to prove the last inequality above, it is enough to set $s_3=e^{s_2}/(\frac{s}{s_1})^s$, $f_s(d)=d^3(s_3)^d$ and note that $f_s(54000) <0.1<  1/(3e)$ while $f_s$ is decreasing for $d$ at least $3/\ln (s_3^{-1})<4613$, which follows by a standard analysis of the first derivative of $f_s$. By~\eqref{dS*Conditional},
\begin{equation}\label{ProbSv}
 \mathbb{P}\left(S_v~|~\bigcap_{B\in\mathcal{B}}\overline{B}\right) 
\leq \mathbb{P}\left(d_{S^*}(v) \geq sd~|~\bigcap_{B\in\mathcal{B}}\overline{B}\right)
= \mathbb{P}\left(d_{S^*}(v) \geq sd\right) 
<\frac{1}{e3d^3}.
\end{equation}

The analysis concerning $R_v$ is analogous. Set
$$R^*=\left\{vw\in E: 1\leq |O_v-O_w|_K\leq \frac{sd+7}{2} \vee 1\leq |I_v-I_w|_K\leq \frac{sd+7}{2}\right\},$$
where $R\subseteq R^*$.
Again the event: $d_{R^*}(v)\geq rd$ is determined by the choices of $C_w$ for vertices $w$ at distance at most $1$ from $v$. Hence, by~\eqref{GammaDefinition} and~\eqref{DeterminationOfAv}, this event is mutually independent of $\mathcal{B}$.
For a fixed colour $C_v$ of $v$, each edge $vw\in E(v)$ is independently included in $R^*(v)$ with probability 
$$2\cdot\left(\frac{2\cdot\left\lfloor\frac{sd+7}{2}\right\rfloor}{K}\right)
-\left(\frac{2\cdot\left\lfloor\frac{sd+7}{2}\right\rfloor}{K}\right)^2.$$ 
Thus, as the function $f(x)=2x-x^2$ is increasing for $x\leq1$ and 
$2\cdot\lfloor\frac{sd+7}{2}\rfloor/K\leq \frac{sd+7}{K}\leq \frac{sd+7}{kd} = \frac{s}{k}+\frac{7}{kd}<0.1292$, 
this probability is bounded above by $f(0.1292)<0.242$. Hence, for any fixed $C\in [K]^2$,
$$\mathbb{E}\left(d_{R^*}(v)~|~C_v=C\right) < 0.242d = \mu_r.$$
Set $r_1=0.242$, $r_2=r-r_1=0.018$ and $\delta_r=\frac{r_2}{r_1}$.
Thus, $rd=(1+\delta_r)\mu_r$ and by the Chernoff Bound, for any $C\in[K]^2$, 
\begin{eqnarray}
\mathbb{P}\left(d_{R^*}(v) \geq rd~|~C_v=C\right) 
&=& \mathbb{P}\left(d_{R^*}(v)\geq (1+\delta_r)\mu_r~|~C_v=C\right) \nonumber\\
&\leq&  \left(\frac{e^{\delta_r}}{(1+\delta_r)^{1+\delta_r}}\right)^{\mu_r}
= \frac{e^{r_2d}}{\left(\frac{r}{r_1}\right)^{rd}}  \nonumber\\
&=& \left(\frac{e^{r_2}}{\left(\frac{r}{r_1}\right)^{r}}\right)^d 
< \frac{1}{e3d^3}. \label{dR*Conditional}
\end{eqnarray}
Again, the last inequality above can be proven via setting $r_3=e^{r_2}/(\frac{r}{r_1})^r$, $f_r(d)=d^3(r_3)^d$ and noting that $f_r(54000) <0.1<  1/(3e)$ while $f_r$ is decreasing for $d$ at least $3/\ln (r_3^{-1})<4592$.
By~\eqref{dR*Conditional},
\begin{equation}\label{ProbRv}
 \mathbb{P}\left(R_v~|~\bigcap_{B\in\mathcal{B}}\overline{B}\right) 
\leq \mathbb{P}\left(d_{R^*}(v) \geq rd~|~\bigcap_{B\in\mathcal{B}}\overline{B}\right)
= \mathbb{P}\left(d_{R^*}(v) \geq rd\right) 
<\frac{1}{e3d^3}.
\end{equation}

The analysis concerning $U_v$ shall be different in at least two aspects.
First of all, for the sake of optimization of the lower bound for $d$ within our argument, we shall not require the investigated event to be mutually independent of $\mathcal{B}$ this time.
We however could not avoid some dependencies even if we did not strive to optimize our approach.
Thus, we are moreover forced to use a different concentration tool, i.e. the McDiarmid's Inequality. In order to facilitate it efficiently, we shall have to be  potentially more wasteful than above. Namely, while estimating $d_U(v)$ for the given vertex $v$ we shall excessively take into account every neighbour of $v$ whose colour is repeated in the neighbourhood of $v$, even if there is no edge between the corresponding uniformly coloured neighbours of $v$. Let us formally define the following superset of $U(v)$:
$$U^*(v)=\{w\in N_G(v)~|~ \exists w'\in N_G(w)\cup N_G(v)\smallsetminus\{w\}: C_{w'}=C_{w}\}.$$
Note that by~\eqref{GammaDefinition} and~\eqref{DeterminationOfAv}, the events in $\mathcal{B}$ are determined by colours assigned to vertices at distance at least $2$ from $v$. In order to bound the probability $\mathbb{P}(|U^*(v)|\geq ud~|~\bigcap_{B\in\mathcal{B}}\overline{B})$ it is thus sufficient to provide some universal bound conditioning on any set of colours assigned to all vertices at distance at least $2$ from $v$ in $G$. We shall additionally assume that the colour of $v$ is arbitrarily fixed. Thus let us consider any event $A$ fixing any specific colours for all vertices in $V\smallsetminus N_G(v)$ (such that $A\subseteq \bigcap_{B\in\mathcal{B}}\overline{B}$).

Let us denote the neighbours of $v$ by $w_1,\ldots,w_d$. For the fixed $A$, the event: $|U^*(v)|\geq ud$ depends on $d$ independent random variables $W_1,\ldots, W_d$ assigning colours to $w_1,\ldots,w_d$, respectively, uniformly at random from $[K]^2$.
 Let $Z_i$ be a binary variable which takes value $1$ if the colour of $w_i$ appears at least twice in the neighbourhood of $v$ (regardless of adjacency of the corresponding neighbours of $v$) or it is the same as the colour of any neighbour of $w_i$ outside $N_G(v)$; $Z_i$ is valued $0$ otherwise. Let 
$$Z=\sum_{i=1}^dZ_i.$$
Note $Z=|U^*(v)|$.

We shall first bound $\mathbb{E}(Z)$. Consider any given $w_i$. Suppose $A$ fixes exactly $d'\leq d$ distinct colours in $N_G(w_i)\smallsetminus N_G(v)$.
Then $Z_i$ shall be valued $0$ if $w_i$ is assigned a colour $C$ distinct from all the $d'$ colours in $N_G(w_i)\smallsetminus N_G(v)$ and moreover, no vertex in $N_G(v)\smallsetminus\{w_i\}$ is assigned $C$. Thus, as $1-x\geq e^{-\frac{x}{1-x}}$ for $0\leq x<1$ and $e^{-x}\geq 1-x$,

\begin{eqnarray}
\mathbb{P}\left(Z_i=0~|~A\right)
&=& \frac{K^2-d'}{K^2}\cdot \left(1-\frac{1}{K^2}\right)^{d-1} 
\geq \frac{K^2-d}{K^2}\cdot e^{-\frac{1}{K^2-1}(d-1)} \nonumber\\
&\geq& \left(1-\frac{d}{K^2}\right) \cdot \left(1-\frac{d-1}{K^2-1}\right)
\geq \left(1-\frac{d}{K^2}\right)^2  \nonumber\\
&\geq& \left(1-\frac{d}{(kd)^2}\right)^2
= 1-\left(\frac{2}{k^2d}-\frac{1}{k^4d^2}\right). \nonumber
\end{eqnarray}
Thus, 
$$\mathbb{E}\left(Z_i~|~A\right) = \mathbb{P}\left(Z_i=1~|~A\right) = 1 - \mathbb{P}\left(Z_i=0~|~A\right) 
\leq \frac{2}{k^2d}-\frac{1}{k^4d^2},$$
and therefore,
$$\mathbb{E}\left(Z~|~A\right)\leq d\cdot\left(\frac{2}{k^2d}-\frac{1}{k^4d^2}\right) <0.059 d.$$
Note that changing the value of any single $W_i$, i.e. the colour of $w_i$, may change the value of $Z$ by at most $c_i=2$. Therefore, by Theorem~\ref{McDiarmidTheorem}, for $u_1=0.059$ and $u_2=u-u_1=0.072$,
\begin{eqnarray}
\mathbb{P}\left(Z\geq ud~|~A\right) 
&=& \mathbb{P}\left(Z\geq u_1d+u_2d~|~A\right)  \nonumber\\
&\leq& \mathbb{P}\left(Z\geq \mathbb{E}(Z~|~A)+u_2d~|~A\right)  \nonumber\\
&\leq& e^{-\frac{\left(u_2d\right)^2}{2\sum_{i=1}^d 2^2}}
= \left(e^{-\frac{(u_2)^2}{8}}\right)^d
< \frac{1}{e3d^3}. \label{dU*Conditional}
\end{eqnarray}
Once again, the last inequality above can be proven via setting $u_3=e^{-(u_2)^2/8}$, $f_u(d)=d^3(u_3)^d$ and noting that $f_u(54000) < 0.11<  1/(3e)$ 
while $f_u$ is decreasing for $d$ at least $3/\ln (u_3^{-1}) =24/(u_2)^2< 4630$.
As $|U^*(v)| =Z$, by~\eqref{dR*Conditional} we thus obtain that:
\begin{equation}\label{ProbUv}
 \mathbb{P}\left(U_v~|~\bigcap_{B\in\mathcal{B}}\overline{B}\right) 
\leq \mathbb{P}\left(|U^*(v)| \geq ud~|~\bigcap_{B\in\mathcal{B}}\overline{B}\right)
<\frac{1}{e3d^3}.
\end{equation}

By~\eqref{GammaEstimation}, \eqref{ProbSv}, \eqref{ProbRv} and~\eqref{ProbUv} the claim follows by the Lov\'asz Local Lemma, i.e. Theorem~\ref{LLL-symmetric-lopsided}.
\end{proof}

\subsection{Decomposition of $G$ to two similar subgraphs}

Fix any vertex colours assignment consistent with (i) -- (iii) in Claim~\ref{MainProbClaim}. 
Basing on this we shall now colour the edges of $G$ with $0$ and $1$ according to the following rules.
Note each of these regards a different subset of edges, which together partition $E$.
\begin{itemize}
\item[$(0^\circ)$] We colour $0$ any edge $vw\in S$ such that $I_v=I_w$ (i.e. $O_v\neq O_w$).
\item[$(1^\circ)$] We colour $1$ any edge $vw\in S$ such that $O_v=O_w$ (i.e. $I_v\neq I_w$).
\item[$(2^\circ)$] The subgraph induced by the edges in $R'\cup U^e$ we colour according to Lemma~\ref{Alon-Wei_lemma}  using a constant function $z\equiv 1/2$.
\item[$(3^\circ)$] The subgraph induced by the edges in $T\smallsetminus U^e$ we colour according to Lemma~\ref{Alon-Wei_lemma}  using a constant function $z\equiv 1/2$.
\item[$(4^\circ)$] The subgraph induced by the (remaining) edges in $E''$
we colour according to Lemma~\ref{Alon-Wei_lemma}  using a constant function $z\equiv 1/2$.
\end{itemize}

For $i=0,1$, we denote by $E_i,U^e_i,T_i,S_i,R_i,R'_i,E'_i,E''_i$ the subsets of edges of $E,U^e,T,S,R,R',E',E''$, respectively, coloured $i$.
We also denote 
$$H_0=(V,E_0),\hspace{1.0cm}H_1=(V,E_1)$$ 
graphs making up a decomposition of $G$.

We shall show that $H_0$ may be further decomposed to two locally irregular subgraphs $F_1$ and $F_2$. An analogous decomposition shall exist for $H_1$, as it has exactly the same features as listed below for $H_0$.

Set 
$$H'_0=(V\smallsetminus U,E'_0)$$
and note that by Rules $(2^\circ)$--$(4^\circ)$ and Lemma~\ref{Alon-Wei_lemma}, for every $v\in V$,
\begin{eqnarray}
|d_{R'_0}(v)-\frac{1}{2}d_{R'}(v)| &\leq& 1, \label{R'0DegClose}\\
|d_{U^e_0}(v)-\frac{1}{2}d_{U^e}(v)| &\leq& 1, \label{Ue0DegClose}\\
|d_{T_0\smallsetminus U^e_0}(v)-\frac{1}{2}d_{T\smallsetminus U^e}(v)| &\leq& 1, \label{T0Ue0DegClose}\\
|d_{E''_0}(v)-\frac{1}{2}d_{E''}(v)| &\leq& 1. \label{E''0DegClose}
\end{eqnarray}

As for every $v\in U$, $d=d_{G}(v) = d_{T}(v) = d_{U^e}(v)+d_{T\smallsetminus U^e}(v)$ and $d_{H_0}(v) = d_{T_0}(v) = d_{U^e_0}(v)+d_{T_0\smallsetminus U^e_0}(v)$, then by~\eqref{Ue0DegClose} and~\eqref{T0Ue0DegClose},
\begin{equation}\label{degree_in_U}
d_{H_0}(v)  \in \left[\frac{d}{2}-2,\frac{d}{2}+2\right]\hspace{0.5cm}{\rm for}\hspace{0.5cm}v\in U.
\end{equation}
Note also that by~\eqref{Ue0DegClose}, \eqref{T0Ue0DegClose} and Claim~\ref{MainProbClaim}(iii),
\begin{eqnarray}
d_{U^e_0}(v) < \frac{ud}{2}+1\hspace{0.5cm}&{\rm for}&\hspace{0.5cm}v\in U, \label{Ue0BoundInU} \\
d_{T_0\smallsetminus U^e_0}(v) > \frac{d-ud}{2}-1\hspace{0.5cm}&{\rm for}&\hspace{0.5cm}v\in U. \label{T0Ue0BoundInU} 
\end{eqnarray}

For every $v\in V\smallsetminus U$, the edges incident with $v$ in 
$G$ can be partitioned as follows: 
$E(v)=S(v)\cup R'(v)\cup T(v)\cup E''(v)$, where $T(v)=(T\smallsetminus U^e)(v)$.
Thus, analogously as above, by~\eqref{R'0DegClose}, \eqref{T0Ue0DegClose} and~\eqref{E''0DegClose},
\begin{equation}\label{degree_outside_U_1}
d_{E_0\smallsetminus S_0}(v) \in \left[\frac{d-d_S(v)}{2}-3,\frac{d-d_S(v)}{2}+3\right].
\end{equation}
As $d_{E_0\smallsetminus S_0}(v) = d_{H_0}(v) - d_{S_0}(v)$, where $S_0\subseteq S$, and hence $0\leq d_{S_0}(v) \leq d_S(v)$, by~\eqref{degree_outside_U_1},
$$d_{H_0}(v) \in \left[\frac{d-d_S(v)}{2}-3,\frac{d+d_S(v)}{2}+3\right].$$
By Claim~\ref{MainProbClaim}(i) we thus obtain that
\begin{equation}\label{degree_outside_U_3}
d_{H_0}(v) \in \left(\frac{d-sd}{2}-3,\frac{d+sd}{2}+3\right)\hspace{0.5cm}{\rm for}\hspace{0.5cm}v\in V\smallsetminus U.
\end{equation}

Note also that for every $v\in V\smallsetminus U$, by~\eqref{E''0DegClose}, Claim~\ref{MainProbClaim} and~\eqref{OptimizedMainConstants},
\begin{eqnarray}
d_{H'_0}(v) &=& d_{E'_0}(v) \geq d_{E''_0}(v) \geq \frac{d_{E''}(v)}{2}-1 = \frac{d_{E\smallsetminus (S\cup T\cup R)}(v)}{2}-1 \nonumber\\
&\geq& \frac{d - d_S(v) - d_T(v) - d_R(v)}{2}-1 > \frac{d-sd-ud-rd}{2}-1 \nonumber\\
&\geq& 12kd+12>12K. 
\label{12K6Lambda}
\end{eqnarray}

\subsection{A decomposition of $H_0$ to two locally irregular subgraphs}

Let 
\begin{equation}\label{d1definition}
d_1 = \frac{d}{6} -\frac{sd}{3}- \frac{ud}{6}   - \frac{13}{3}
\end{equation}
and 
\begin{equation}\label{DDefinition}
D=\{0,1,\ldots,\lceil d_1\rceil-1\}.
\end{equation}
Note that $|D|\geq d_1$. 
By~\eqref{OptimizedMainConstants},
$$\frac{ud}{2}+1\leq d_1\leq \frac{d-ud}{2}-1.$$
Thus, by~\eqref{Ue0BoundInU} and~\eqref{T0Ue0BoundInU}, 
\begin{equation}\label{ConstraintsU-D}
d_{U^e_0}(v)<|D|\hspace{0.5cm}{\rm and}\hspace{0.5cm} d_{T_0\smallsetminus U^e_0}(v)> \lceil d_1\rceil -1\hspace{0.5cm} {\rm for}\hspace{0.5cm} v\in U.
\end{equation}

For every vertex  $v\in U$ we choose any subset $T_v$ of the set of edges $T_0(v)\smallsetminus U^e_0(v)$ with 
\begin{equation}\label{TvInD}
|T_v|\in D
\end{equation}
 so that 
\begin{equation}\label{DistinctDegF2U1}
d_{H_0}(v)-|T_v| \neq d_{H_0}(w)-|T_w|\hspace{0.5cm} {\rm for~ every}\hspace{0.5cm} vw\in U^e_0.
\end{equation}
By~\eqref{ConstraintsU-D} we may do this greedily, analysing one vertex in $U$ after another. 
 
Let 
$$T_U=\bigcup_{v\in U}T_v.$$
Set 
\begin{equation}\label{LambdaDefinition}
\lambda = 2K.
\end{equation}
With every  $v\in V\smallsetminus U$ we associate an integer $t(v)\in \{0,1,\ldots,\lambda-1\}$ such that
\begin{eqnarray}\label{MainT(v)Condition}
d_{T_U\cup R_0}(v)+t(v)\equiv 2O_v \pmod{\lambda}. 
\end{eqnarray}
Note that by~\eqref{12K6Lambda} and~\eqref{LambdaDefinition}, $d_{H'_0}(v)>6\lambda$ for every $v\in V\smallsetminus U$. 
Thus we may apply Lemma~\ref{1_6Lemma} to $H'_0$ with $\lambda_v=\lambda$ for $v\in V\smallsetminus U$ and $t:V\smallsetminus U \to \{0,1,\ldots,\lambda-1\}$ fulfilling~\eqref{MainT(v)Condition}.
The resulting subgraph of $H'_0$ we denote by $H$. 
Thus, by Lemma~\ref{1_6Lemma}  and~\eqref{MainT(v)Condition}, for every $v\in V\smallsetminus U$,
\begin{eqnarray}
d_H(v)&\in&\left[\frac{d_{H'_0}(v)}{3},\frac{2d_{H'_0}(v)}{3}\right], \label{dHv1323}\\
d_H(v)&\equiv& t(v), t(v)+1 \pmod {\lambda}, \label{dHvMod}
\end{eqnarray}
where the latter means that $d_H(v)$ is equivalent either to $t(v)$ or to  $t(v)+1$ modulo $\lambda$. 
We finally define a decomposition of $H_0$ to $F_1=(V,E(F_1))$ and $F_2=(V,E(F_2))$ by setting:
\begin{eqnarray}
E(F_1) &=& T_U\cup R_0\cup E(H), \label{EF1Def} \\
E(F_2) &=& E_0\smallsetminus E(F_1) = (E'_0\smallsetminus E(H))\cup (T_0\smallsetminus T_U). \label{EF2Def}
\end{eqnarray}

Note that by~\eqref{MainT(v)Condition}, \eqref{dHvMod} and~\eqref{EF1Def},
for every  $v\in V\smallsetminus U$,
\begin{eqnarray}\label{MainF1ModCondition}
d_{F_1}(v) \equiv 2O_v, 2O_v+1 \pmod{\lambda}.
\end{eqnarray}

\subsection{Justification of no degree conflicts in $F_1$ and $F_2$}

Consider any edge $vw\in E_0$. By~\eqref{EF2Def} it must be an edge of  either $F_1$ or $F_2$.
We shall exhibit that in each of these cases its ends shall have distinct degrees in the corresponding graph.

Suppose first that $v,w\in U$, i.e. $vw\in U^e_0$. Then, by~\eqref{EF2Def}, $vw\in F_2$ and 
$d_{F_2}(v) = d_{H_0}(v)-|T_v|$,  $d_{F_2}(w)=d_{H_0}(w)-|T_w|$.  
Thus, by~\eqref{DistinctDegF2U1}, $d_{F_2}(v) \neq d_{F_2}(w)$.

Suppose next that $v,w\in V\smallsetminus U$. Then, by Rule $(1^\circ)$, $O_v\neq O_w$.

If $vw\in E(F_1)$, this and~\eqref{MainF1ModCondition} immediately imply that
$d_{F_1}(v)\neq d_{F_1}(w)$ (even modulo $\lambda=2K$).

Otherwise, if $vw\in E(F_2)$, then by~\eqref{EF1Def} and~\eqref{EF2Def}, $vw\notin R$. Thus, 
$|O_v-O_w|_K > \frac{sd+7}{2}$ (as $O_v\neq O_w$).
Hence, $|2O_v-2O_w|_{2K} > sd+7$. Therefore, by~\eqref{MainF1ModCondition},
\begin{equation}\label{VUVUF2-C1}
|d_{F_1}(v)-d_{F_2}(w)|>sd+6.
\end{equation}
Moreover, by~\eqref{degree_outside_U_3}, 
\begin{equation}\label{VUVUF2-C2}
|d_{H_0}(v)-d_{H_0}(w)|<sd+6.
\end{equation} 
Since $d_{F_2}(v)=d_{H_0}(v)-d_{F_1}(v)$ and  $d_{F_2}(w)=d_{H_0}(w)-d_{F_1}(w)$,
then~\eqref{VUVUF2-C1} and~\eqref{VUVUF2-C2} imply that $d_{F_2}(v)\neq d_{F_2}(w)$.

It thus remains to consider the case when $v\in V\smallsetminus U$ and $w\in U$. 
Note that then, by~\eqref{EF1Def}, \eqref{DDefinition} and~\eqref{TvInD}, we have 
\begin{equation}\label{dF1InUUpperD1}
d_{F_1}(w)=|T_w|<d_1.
\end{equation}

Suppose first that $vw\in F_1$.
Note that by~\eqref{EF1Def}, \eqref{dHv1323}, \eqref{R'0DegClose}, \eqref{E''0DegClose}, Claim~\ref{MainProbClaim} and~\eqref{d1definition},
\begin{eqnarray}
d_{F_1}(v) &\geq& d_{R_0}(v) + d_H(v)  
\geq d_{R_0}(v) + \frac{1}{3}d_{H'_0}(v)  \nonumber\\
&\geq& \frac{1}{3} \left(d_{R_0}(v)+d_{H'_0}(v)\right) 
= \frac{1}{3} d_{E_0\smallsetminus T_0}(v) \nonumber \\
&\geq& \frac{1}{3} d_{E_0\smallsetminus (S_0\cup T_0)}(v)
= \frac{1}{3} \left(d_{R'_0}(v)+d_{E''_0}(v)\right) \nonumber\\
 &\geq& \frac{1}{3} \left(\frac{d_{R'}(v)}{2}-1+\frac{d_{E''}(v)}{2}-1\right)  
 =  \frac{1}{3} \left(\frac{d_{E\smallsetminus(S\cup T)}(v)}{2}-2\right) \nonumber\\
&>& \frac{1}{3} \left(\frac{d-sd-ud}{2}-2\right) 
= \frac{d}{6}-\frac{sd}{6}-\frac{ud}{6}-\frac{2}{3}  >d_1. \label{DF1InVULower}
\end{eqnarray}
Hence, by~\eqref{dF1InUUpperD1} and~\eqref{DF1InVULower},  $d_{F_1}(w) < d_{F_1}(v)$.

Assume finally that $vw\in F_2$.
Then, by~\eqref{EF2Def}, \eqref{dF1InUUpperD1}, \eqref{degree_in_U} and~\eqref{d1definition},
\begin{eqnarray}
d_{F_2}(w) &=& d_{H_0}(w)-d_{F_1}(w)> d_{H_0}(w)-d_1
\geq \frac{d}{2}-2-d_1  \nonumber\\
&=& \frac{d}{3} + \frac{sd}{3} + \frac{ud}{6}+\frac{7}{3}. \label{dF2wInULowerBound}
\end{eqnarray}

On the other hand, by~\eqref{EF2Def}, \eqref{dHv1323}, \eqref{degree_outside_U_3}, \eqref{T0Ue0DegClose} and Claim~\ref{MainProbClaim}(iii),
\begin{eqnarray}
d_{F_2}(v) &=& d_{T_0}(v)-d_{T_U}(v) + d_{H'_0}(v)-d_H(v) \nonumber\\
&\leq& d_{T_0}(v)+\frac{2}{3}d_{H'_0}(v) 
\leq d_{T_0}(v) + \frac{2}{3}\left(d_{H_0}(v)-d_{T_0}(v)\right) \nonumber\\
&=& \frac{2}{3}d_{H_0}(v) + \frac{1}{3}d_{T_0}(v) 
< \frac{2}{3}\left(\frac{d+sd}{2}+3\right) + \frac{1}{3}\left(\frac{ud}{2}+1\right) \nonumber\\
&=& \frac{d}{3}+\frac{sd}{3}+\frac{ud}{6}+\frac{7}{3}. \label{dF2vInVUUpper}
\end{eqnarray}
Therefore, by~\eqref{dF2wInULowerBound} and~\eqref{dF2vInVUUpper}, $d_{F_2}(v)<d_{F_2}(w)$.

As adjacent vertices have distinct degrees in $F_1$ and $F_2$, these graphs are indeed locally irregular.

This finishes the proof of Theorem~\ref{MainResult-4-decomposability_of_regular}, as a decomposition of $H_1$ to two locally irregular subgraphs can be performed in exactly the same manner
as the decomposition of $H_0$ to $F_1$ and $F_2$. \qed

\section{Final remarks}.

There are several ways one may diverse our approach, deviating here and there in alternative directions.
Our primary goal was however optimization of the lower bound for $d$ in Theorem~\ref{MainResult-4-decomposability_of_regular}. We thus finally decided to present a variant of our approach which yields 
the best result in this regard.
On the other hand, we wanted to keep balance between potential further minor reductions of the mentioned bound and maintaining clarity of presentation of our approach and its accessibility. 
We are thus aware that our result could still be somewhat improved via more tedious calculations or more complex construction and analysis.
A potentially largest gain could additionally be attained due to applications of new results concerning stochastic processes, which imply improved variants of McDiarmid's Ineqality.

Our secondary goal was a clear presentation of our approach, which we reckon interesting on its own, hoping its elements may be useful in other contexts. In particular, we believe it can be developed towards obtaining new results concerning decomposability to locally irregular subgraphs of  general graphs with degrees large enough, not only regular ones. Several technical obstacles must however still be overcome to that end. 

Let us finally mention we could easily perform our proof without using Lemma~\ref{Alon-Wei_lemma} of Alon and Wei, and exploit the classical Euler's Theorem instead. It was however very useful in an another variant of our argument, which was eventually altered for the sake of optimizing the lower bound for $d$. We however decided to retain Lemma~\ref{Alon-Wei_lemma} in the paper, as it indeed gives many options for possible applications in this and related problems, and can potentially be used while developing further some of our ideas.\\

Declarations of interest: none.

\end{document}